\documentclass[english]{jaums}
\usepackage[T1]{fontenc}
\usepackage[latin9]{inputenc}
\usepackage{graphicx}

\makeatletter
 
\newtheorem{thm}{theorem}

\newtheorem{defn}{definition}
\newtheorem{lem}{lemma}

\usepackage{babel}
\makeatother

\begin{document}

\title{Free actions and Grassmanian variety}

\author{Dr Burzin Bhavnagri}

\begin{abstract}
An algebraic notion of representational consistency is defined. A
theorem relating it to free actions is proved. A metrizability problem
of the quotient (a shape space) is discussed. This leads to a new
algebraic variety with a metrizability result. A concrete example
is given from stereo vision. 
\end{abstract}


\maketitle

\section{Introduction}

One of the famous paradoxes of Zeno is Achilles and the tortoise.
This paradox can be resolved by the continuum hypothesis, from which
is proved the principle of mathematical induction. Because of the
principle of induction it is difficult to construct properties which
hold for a set but which no members can possess individually. 

This paper introduces an algebraic notion of representational consistency
from Computer Vision using Shape Spaces \cite{bhav96}. An existence
theorem will be stated. We will construct an example of one of the
existence conditions. But the other existence condition can only be
met if the original set is split. This is intimately connected with
human vision, pattern recognition and neural networks. Even this was
a problem of considerable difficulty since it had been proved there
are no general case view invariants \cite{burns-etal92} as was previously
mentioned in \cite{le-bhavnagri97}. The problem of constructing a
non-trivial example of representational consistency without decomposing
the original set remains unsolved. 

The idea bears a little resemblance to the Chinese Room paradox of
John Searle \cite{searle}. Suppose you are in a room in which you
receive Chinese characters, consult a book containing an English version
of the aforementioned computer program and process the Chinese characters
according to its instructions. You do not understand a word of Chinese;
you simply manipulate what, to you, are meaningless symbols, using
the book and whatever other equipment, like paper, pencils, erasers
and filing cabinets, are available to you. After manipulating the
symbols, you respond to a given Chinese question in the same language.
As the computer passed the Turing test \cite{penrose89,penrose94}this
way, it is fair, says Searle \cite{searle}, to deduce that he has
done so, too, simply by running the program manually. \char`\"{}Nobody
just looking at my answers can tell that I don't speak a word of Chinese\char`\"{}.

The idea we will use is this. The fingers of my hand can represent
one, two, three, four and five; but they can also represent two, one,
three, four and five respectively. Both these representations exist
simultaneously, so we can pair them up to give the pair one,two repeated
twice, and three, four and five, which no longer has the structure
of five elements. From these two representations which exist simultaneously,
a third can be constructed also. But this third representation has
four elements not five, it loses the structure of five elements. 

I call this idea representational consistency both in analogy and
to distinguish it from logical consistency, and because the existence
theorem is based on the idea of representation theory. Since I first
came across an article about the Chinese Room paradox, I felt puzzled
about my own paradox for many years. I had been discovering a number
of new theorems since I was an intern at Monash University with Alan
Pryde. After my honours thesis with John Stillwell (who is well known
for his books like \cite{stillwell80,stillwell92}) I published them
in \cite{bhav94}. These theorems concern a method for representing
shape based on an equivalence relation on polygons \cite{bhav94}.
Now commonly known as a binary string descriptor \cite{tung2000,tung97b,tungetal97,cheng98,kinglau,zles2006}.
My doctoral supervisor Alan Carey noticed a problem, something nobody
else has noticed even in the ensuing decade, and pointed me in the
direction of Algebraic Geometry \cite{acarey91,acarey96b,acarey99b,acarey2000,acarey2006,acarey2007}.
My other doctoral supervisor Michael Brooks introduced me to computer
vision \cite{horn-brooks86,horn-brooks89}, particularly stereo vision.
I then proved the theorem in \cite{bhav95d,bhav95e,bhav97} . Some
years later, in an attempt to generalize the binary string descriptor,
I discovered that the paradox has a mathematical resolution.

\section{Representational consistency}

\begin{defn}
Let $S$ be a set that satisfies a collection of relations, called
the structure of $S$. A \textbf{representation} of $S$ is a one
to one mapping from $S$ to some other space $I$, which preserves
the structure of $S$.
\end{defn}
\begin{equation}
{\displaystyle R:S}\rightarrow I\end{equation}

$\vphantom{}$

The mathematical resolution to the paradox is provided by constructing
a mapping analogous to the many mappings like the Plucker map and
Veronese map that can be found in Algebraic Geometry \cite{harris92}. 

\begin{defn}
Let $R_{1},\: R_{2}$ be two representations of $S$

\begin{equation}
\cup(R_{1},\: R_{2})=\left\{ \left(s,\left\{ R_{1}(s)\right\} \cup\left\{ R_{2}(s)\right\} \right)\mid s\in S\right\} \end{equation}
This defines an operator on mappings whose domain is $S$ but it is
not a linear operator. The representations are pairwise consistent
if for all pairs of representations $R_{1},\: R_{2}$ the map $\cup(R_{1},R_{2})$
is one to one. 
\end{defn}
$\vphantom{}$

The following is a basic existence theorem. The existence follows
by taking the contrapositive.

\begin{thm}
Let the automorphisms of $S$ be denoted $Aut(S)$, and suppose they
are a group under composition, which acts freely on $S$ . If the
representations of S are pairwise inconsistent then $Aut(S)$ has
elements that are involutions.
\end{thm}
\begin{proof}
If the representations of $S$ are pairwise inconsistent then for
some $\phi_{1},\phi_{2}$ $\cup(\phi_{1},\phi_{2})$ is not one to
one.

\begin{equation}
\exists s,t\in S\: s\neq t\:\left\{ \phi_{1}(s),\phi_{2}(s)\right\} =\left\{ \phi_{1}(t),\phi_{2}(t)\right\} \end{equation}

Since $\phi_{1},\phi_{2}$ are representations they are one to one
so

\begin{equation}
\phi_{1}(s)\neq\phi_{1}(t)\quad\phi_{2}(s)\neq\phi_{2}(t)\end{equation}

Since $Aut(S)$ is a group we have

\begin{equation}
\phi_{2}^{-1}\left(\phi_{1}(s)\right)=t\end{equation}

\begin{equation}
\phi_{1}^{-1}\left(\phi_{2}(s)\right)=t\end{equation}

Hence $\phi_{1}^{-1}\phi_{2}\phi_{1}^{-1}\phi_{2}(s)=s$

Since $Aut(S)$ acts freely on $S$, $id\in Aut(S)$ $\phi_{1}^{-1}\phi_{2}\phi_{1}^{-1}\phi_{2}=id$

$\vphantom{}$

Since $\cup(\phi_{1},\phi_{2})$ is not one to one and $\phi_{1},\phi_{2}$
are one to one $\phi_{1}\neq\phi_{2}$

\begin{equation}
\left(\phi_{1}^{-1}\phi_{2}\right)^{2}=id\end{equation}
 The group $Aut(S)$ has an involution.
\end{proof}

\section{Group action and non-metrizability}

Astrom \cite{astrom95} had shown there are distinctively shaped curves
which can be mapped arbitrarily close to a circle by projective transformations,
such as sillouettes of a duck and a rabbit. It turned out to be related
to free actions, or the lack thereof. 

The set of $m$ dimensional linear subspaces of $R^{n}$ is denoted
$G(m,n)$ and is called a Grassmanian.

$\vphantom{}$

Let $G$ be the set of $(g,d)$ such that $g\in GL(4)$ and $d$ is
$n\times n$ diagonal.

$\vphantom{}$

The action of $GL(4)\times diag(GL(n))$ is not free

\begin{equation}
A\mapsto gAd\end{equation}

\emph{Example}

\[
A=\left[\begin{array}{cccccc}
1 &  &  &  & 1 & 1\\
 & 1 &  &  & 1 & 2\\
 &  & 1 &  & 1 & 3\\
 &  &  & 1 & 1 & 5\end{array}\right]\]

\[
g=\left[\begin{array}{cccc}
\alpha\\
 & \beta\\
 &  & \gamma\\
 &  &  & \delta\end{array}\right]\]

\[
d=diag(\alpha^{-1},\beta^{-1},\gamma^{-1},\delta^{-1},1,1)\]

\[
gAd=A\]

Michael Murray (who is well known for his book on differential geometry
and statistics \cite{murray93}) brought to my attention the literature
on Shape Spaces. He also suggested a trick by which I might be able
to prove the theorem below. 

\emph{Theorem}

$M=\{m\in M_{4n}\mid\exists4\times4\: minor\neq0\:\forall i\: M_{4i}\neq0\}$

$M/\left(GL(m)\times diag(GL(n))\right)$ is not Hausdorff

Consequently no metric exists on $M/\left(GL(m)\times diag(GL(n))\right)$

\section{Restricted group action and metrizability}

Let $G$ be the set of $(g,d)$ such that

\begin{equation}
g\in GL(4)\quad g_{41}=g_{42}=g_{43}=0\quad\mathrm{and}\quad g_{44}=1\label{eq:4-1}\end{equation}

\begin{equation}
d\in diag(GL(n))\quad n\geq4\label{eq:4-2}\end{equation}

\emph{Theorem}$\quad$If $A\in M$ and $(g,d)\in G$

\begin{equation}
gAd=A\Rightarrow g=Id\:\mathrm{and}\: d=Id\label{eq:4-3}\end{equation}

This restricted group action acts freely.

\begin{equation}
\iota:\widetilde{M}\rightarrow M_{4,n}\quad\left[\begin{array}{cccc}
x & y & z & t\end{array}\right]^{T}\mapsto\left[\begin{array}{cccc}
\frac{x}{z} & \frac{y}{z} & 1 & \frac{t}{z}\end{array}\right]^{T}\label{eq:4-4}\end{equation}

\begin{equation}
\widetilde{M}=\left\{ m\in M\mid\forall i\: m_{3i}\neq0\right\} \label{eq:4-5}\end{equation}

\begin{equation}
\forall g:\widetilde{M}\rightarrow\widetilde{M}\quad\exists d\; g\iota(X)d=\iota(gX)\label{eq:4-6}\end{equation}

\emph{Theorem}$\quad$The quotient of $\iota(\widetilde{M})$ by the
restricted action is a manifold and is metrisable

\emph{Sketch}

$X\in\widetilde{M}$

Since $G$ acts freely and $X\mapsto[\iota(X)]$ is continuous and
open there is a neighborhood of the restricted orbit of $\iota(X)$
homeomorphic to a neighborhood of the orbit of $X$ in $G(3,k-1)$

\section{Non-metrizability theorems}

\begin{defn}
A Lie group $H$ is said to act properly on a manifold $M$ if and
only if, for all compact subsets $K$ contained in $M$, $H_{K}=\{g\in H\mid gK\cap K\neq\emptyset\}$
is relatively compact in $H$.
\end{defn}
\begin{thm}
If $H$ is a Lie group acting on a manifold $M$, then $M/H$ is Hausdorff
if and only if $H$ acts properly on $M$.
\end{thm}
\begin{proof}
Suppose $H$ acts properly on $M$ but $M/H$ is not Hausdorff. Consider
$x\in M$, and $\{gy\mid g\in H\}\subset M$. If there do not exist
open sets separating them, then $x$ is in the closure of $\{gy\mid g\in H\}$.
So there must be a sequence $g_{i}y$ such that $x$ is a limit point
of the sequence. The set $K=\{x\}\cup\{g_{i}y\mid i=1\ldots\infty\}$
is compact. Since H acts properly $\{g\in H\mid gK\cap K\neq\emptyset\}$
is also compact. However, the sequence $g_{i}$ has no limit, so it
is not closed in $H$. This contradicts the compactness of $H_{K}$,
so $M/H$ must be Hausdorff.

Now suppose $M/H$ is Hausdorff. Suppose also that $K$ is a compact
subset of $M$. Since $M/H$ is Hausdorff, if $x,y\in K$ belong to
different orbits then there are open sets $U,V$ containing $x,y$
respectively, such that $U\cap gV=\emptyset$ for all $g\in H$. Thus
for each subset of $K$ belonging to a single orbit, we obtain an
open set covering it that is disjoint from all other such open sets.
In other words we obtain an open cover of $K$ with disjoint members.
By definition of compactness, this has a finite subcover, which must
be the same cover, because the members were disjoint. Hence $K$ consists
of a finite number of disjoint subsets $K_{i}$ each belonging to
a single orbit of $H$. Each $K_{i}$ must be compact, otherwise the
compactness of $K$ would be contradicted. Since $H_{K}=\{g\in H\mid gK\cap K\neq\emptyset\}$
is a finite union of an intersection of compact sets, it is compact.
It follows that $H$ acts properly on $M$. 
\end{proof}
\begin{lem}
$GL(4)\times diag(GL(n))$ does not act properly on $({\bf R}^{4}\setminus\{0\})^{n}$
\end{lem}
\begin{proof}
Let $K=O(4)\times\{p\}\subset M_{4,n}$.\\
 Take any $g\in GL(4)$. It has a singular value decomposition $u_{g}d_{g}v_{g}^{\top}$
with $d_{g}\in diag(GL(4))$.\\
 Take $k=\left[\begin{matrix}v_{g}\end{matrix}\right]$, $v_{g}\in O(4)$,
$d\in diag(d_{g}^{-1},1,\ldots,1)$.\\
 Then $gkd=u_{g}d_{g}v_{g}^{\top}\left[\begin{matrix}v_{g}\end{matrix}\right]d=\left[\begin{matrix}u_{g}\end{matrix}\right]$.\\
 So if $p$ is a unit eigenvector of $g$, $gkd\in O(4)\times\{p\}$.\\
 Hence $H_{K}$ contains the subset of $GL(4)$ fixing $p$, which
is non-compact. 
\end{proof}
Since $GL(4)\times diag(GL(n))$ does not act properly\index{proper action}
on $M$, $M/(GL(4)\times diag(GL(n)))$ is not Hausdorff. Consequently
no metric exists on $M/(GL(4)\times diag(GL(n)))$.

\section{Metrizability theorems}

\vskip 10pt In what follows, $GL(m)$ denotes the set of non-singular
$m\times m$ matrices, $Id$ denotes the identity matrix, $diag()$
denotes a diagonal matrix whose diagonal entries are listed inside
the brackets, and $M_{m,n}$ denotes the set of $m\times n$ matrices.

\vskip 10pt Let $M$ be the set of $4\times n$ matrices which have
at least one non-zero four by four minor and no entry of the fourth
row zero. A matrix in $M$ that has all entries of its fourth row
equal to one represents $n\geq4$ points in ${\bf R}^{3}$ with at
least four points not coplanar.

\vskip 5pt Let $G$ be the set of $(g,d)$ such that $g\in GL(4)$
with $g_{41}=g_{42}=g_{43}=0$ and $g_{44}=1$, and $d$ is $n\times n$
diagonal and non-singular where $n\geq4$. $G$ is a group%
\footnote{ $G$ is in fact a Lie group%
} under the product operation $(g_{1},d_{1})\circ(g_{2},d_{2})=(g_{1}g_{2},d_{1}d_{2})$.

\begin{lem}
If $A\in M$, $(g,d)\in G$ then $gAd=A$ implies $g=Id$ and $d=Id$.
\end{lem}
\begin{proof}
First we show this when $n=4$.\\
 Suppose $gAd=A$, where $A,g,d$ satisfy the stated hypotheses.\\
 Then $A$ is invertible since it is $4\times4$ and has a non-zero
$4\times4$ minor.\\
 Multiplying $gAd=A$ by $d^{-1}A^{-1}$ on the right, we obtain $g=Ad^{-1}A^{-1}$,
so $g$ is diagonalizable. \\
 Moreover, $d^{-1}$ is its eigenvalue matrix, and $A$ is an eigenvector
matrix of $g$.\\
 Suppose $v$ and $\lambda$ are such that $gv=\lambda v$.\\
 Then \begin{equation}
\left[\begin{array}{cccc}
g_{41} & g_{42} & g_{43} & 1\end{array}\right]v=\lambda v\label{eq:6-1}\end{equation}
 $\Rightarrow v_{4}=\lambda v_{4}$.\\
 Since columns of $A$ are eigenvectors of $g$ and no fourth row
entry is zero for $A\in M$, $v_{4}\neq0$. Thus $\lambda=1$. So
all eigenvalues of $g$ are equal to $1$. Since $d^{-1}$ is the
eigenvalue matrix of $g$, it is the identity matrix. Since the inverse
of the identity matrix is still the identity matrix, $d=Id$. Hence
also $g=Id$.\\

If $n>4$, $A\in M$ has a non-singular $4\times4$ minor $A_{i}$
whose fourth row has no zero entries, where $i=(i_{1},i_{2},i_{3},i_{4})$.
The columns $i_{1},i_{2},i_{3},i_{4}$ of $gAd$ are $gA_{i}diag(d_{i_{1}},d_{i_{2}},d_{i_{3}},d_{i_{4}})$.
The $n=4$ case implies $g=Id$ and $d_{i_{1}},\ldots,d_{i_{4}}=1$.
Hence also $d=Id$. 
\end{proof}
\vskip 10pt Let $\tilde{M}$ be the subset of $M$ whose elements
(are matrices) with no entry of their third row equal to zero. Let
$\iota$ denote the map from $\tilde{M}$ to $M_{4,n}$ which maps
each column $\left[\begin{array}{cccc}
x & y & z & t\end{array}\right]^{\top}$ to $\left[\begin{array}{cccc}
\frac{x}{z} & \frac{y}{z} & 1 & \frac{t}{z}\end{array}\right]^{\top}$. $\iota$ represents perspective image formation. The fact that no
entry of the third row is zero means that no point in the configuration
represented by a matrix in $\tilde{M}$ lies in the focal plane.

\begin{thm}
\label{igxeqn}If $g\in GL(4)$ with $g_{41}=g_{42}=g_{43}=0$ and
$g_{44}=1$, $X\in\tilde{M}$, $gX\in\tilde{M}$, then there is a
unique non-singular $n\times n$ diagonal matrix $d$ such that \begin{equation}
g\iota(X)d=\iota(gX)\label{eq:6-2}\end{equation}
 
\end{thm}
\begin{proof}
Let \begin{equation}
g=\left[\begin{array}{cccc}
g_{11} & g_{12} & g_{13} & g_{14}\\
g_{21} & g_{22} & g_{23} & g_{24}\\
g_{31} & g_{32} & g_{33} & g_{34}\\
0 & 0 & 0 & 1\end{array}\right]\label{eq:6-3}\end{equation}
 be an affine transformation.

\begin{equation}
X=\left[\begin{array}{ccc}
x_{1} & \cdots & x_{n}\\
y_{1} & \cdots & y_{n}\\
z_{1} & \cdots & z_{n}\\
t_{1} & \cdots & t_{n}\end{array}\right]\in\tilde{M}\label{eq:6-4}\end{equation}
so $z_{1},\ldots,z_{n}\neq0$, and $X$ is not contained in a plane.

\begin{equation}
g\iota(X)=\left[\begin{array}{ccc}
(g_{11}x_{1}+g_{12}y_{1}+g_{13}z_{1}+g_{14}t_{1})/z_{1} & \cdots & (g_{11}x_{1}+g_{12}y_{1}+g_{13}z_{1}+g_{14}t_{1})/z_{n}\\
(g_{21}x_{1}+g_{22}y_{1}+g_{23}z_{1}+g_{24}t_{1})/z_{1} & \cdots & (g_{21}x_{1}+g_{22}y_{1}+g_{23}z_{1}+g_{24}t_{1})/z_{n}\\
(g_{31}x_{1}+g_{32}y_{1}+g_{33}z_{1}+g_{34}t_{1})/z_{1} & \cdots & (g_{31}x_{1}+g_{32}y_{1}+g_{33}z_{1}+g_{34}t_{1})/z_{n}\\
t_{1}/z_{1} & \cdots & t_{n}/z_{n}\end{array}\right]\label{expandgix}\end{equation}

\begin{equation}
\iota(gX)=\left[\begin{array}{ccc}
\frac{g_{11}x_{1}+g_{12}y_{1}+g_{13}z_{1}+g_{14}t_{1}}{g_{31}x_{1}+g_{32}y_{1}+g_{33}z_{1}+g_{34}t_{1}} & \cdots & \frac{g_{11}x_{n}+g_{12}y_{n}+g_{13}z_{n}+g_{14}t_{n}}{g_{31}x_{1}+g_{32}y_{1}+g_{33}z_{1}+g_{34}t_{1}}\\
\frac{g_{21}x_{1}+g_{22}y_{1}+g_{23}z_{1}+g_{24}t_{1}}{g_{31}x_{1}+g_{32}y_{1}+g_{33}z_{1}+g_{34}t_{1}} & \cdots & \frac{g_{21}x_{n}+g_{22}y_{n}+g_{23}z_{n}+g_{24}t_{n}}{g_{31}x_{n}+g_{32}y_{n}+g_{33}z_{n}+g_{34}t_{n}}\\
1 & \cdots & 1\\
\frac{t_{1}}{g_{31}x_{1}+g_{32}y_{1}+g_{33}z_{1}+g_{34}t_{1}} &  & \frac{t_{n}}{g_{31}x_{n}+g_{32}y_{n}+g_{33}z_{n}+g_{34}t_{n}}\end{array}\right]\label{expandigx}\end{equation}

Let \begin{equation}
d=diag(z_{1}/(g_{31}x_{1}+g_{32}y_{1}+g_{33}z_{1}+g_{34}t_{1}),\ldots,z_{n}/(g_{31}x_{n}+g_{32}y_{n}+g_{33}z_{n}+g_{34}t_{n}))\label{eq:6-7}\end{equation}
 Since the third rows of $X$ and $gX$ have no non-zero entries,
d is non-singular. Multiplying the right hand side of equation \ref{expandgix}
by $d$ we see that it equals the right hand side of equation \ref{expandigx}.
Thus $g\iota(X)d=\iota(gX)$ which proves existence.\\

Suppose $(g^{\prime},d^{\prime})\in G$ such that $g^{\prime}\iota(X)d^{\prime}=\iota(gX)$.
Then $g^{\prime}\iota(X)d^{\prime}=\iota(gX)=g\iota(X)d$. Unless
$g^{\prime}=g$ and $d^{\prime}=d$, $(g^{\prime}g^{-1},d^{\prime}d^{-1})$
fixes $\iota(X)$. $\iota(X)=X\lambda$ for some $\lambda\in diag(M_{n,n})$.
Since $X\in\tilde{M}$ its third row has no zero entries, so $\lambda$
is non-singular, so $X\lambda\in M$, ie $\iota(X)\in M$. By lemma
2.2, $g^{\prime}=g$ and $d^{\prime}=d$.\\
 Thus $g$ and $d$ are unique. 
\end{proof}
Consider the following action of the group $G$ on $\iota(\tilde{M})$:
\begin{equation}
\theta:G\times\iota(\tilde{M})\rightarrow\iota(\tilde{M})\quad\theta((g,d),\iota(X))=g\iota(X)d\label{eq:6-8}\end{equation}

We will restrict $\theta$ to those elements of $G\times\iota(\tilde{M})$
that satisfy $g\iota(X)d=\iota(gX)$ and $gX\in\tilde{M}$. We will
call this the restricted action of $G$.

\vskip 10pt Suppose the fourth coordinates $t_{i}$ of all scene
points equal $1$. Then $\iota$ becomes invertible.

\begin{lem}
Let $\iota(X)\sim\iota(Y)$ if there is a $g\in GL(4)$ such that
$g_{41}=g_{42}=g_{43}=0$, $g_{44}=1$ and $\iota(X)=\iota(gY)$,
where $X,Y\in\iota(\tilde{M})$. This is an equivalence relation on
$\iota(\tilde{M})$.
\end{lem}
\begin{proof}
Suppose $\iota(X)=\iota(gY)$, so for each $i$ \begin{equation}
Y_{1i}/Y_{3i}=(gX)_{1i}/(gX)_{3i}\label{eq:6-9a}\end{equation}
 \begin{equation}
Y_{2i}/Y_{3i}=(gX)_{2i}/(gX)_{3i}\label{eq:6-9b}\end{equation}
 \begin{equation}
1/Y_{3i}=1/(gX)_{3i}\label{eq:6-9c}\end{equation}
 Thus $Y=gX$. Since the set of $g\in GL(4)$ such that $g_{41}=g_{42}=g_{43}=0$
and $g_{44}=1$ are a group, $\sim$ is an equivalence relation. 
\end{proof}
\vskip 10pt We can therefore see that the restricted group action
partitions $\iota(\tilde{M})$ into equivalence classes. Thus there
is a quotient space $\iota(\tilde{M})/\sim$. We will introduce two
lemmas from \cite{boothby75} to prove that this quotient space is
a manifold.

\begin{defn}
An equivalence relation on a space $X$ is called open if whenever
a subset $A\subset X$ is open, then $[A]$ is also open.
\end{defn}
\begin{lem}
\label{openrel}The equivalence relation associated with the restricted
group action is open.
\end{lem}
\begin{proof}
Suppose $g_{31}=0,g_{32}=1,g_{33}=g_{34}=0$. Then $d$ is of the
form ${\rm diag}(v/w)$ for vectors $v,w$, where division is componentwise,
and $v\neq\lambda w$ for all non-zero $\lambda$. All $d\neq{\rm diag}(\lambda,\ldots,\lambda)$
can be realized like this. If $g_{31}=1/\lambda,g_{32}=g_{33}=g_{34}=0$
then $d={\rm diag}(\lambda,\ldots,\lambda)$. Hence $d$ is surjective.
Let $D$ be any open subset of the set of non-singular diagonal matrices.
Since $d$ in theorem \ref{igxeqn} is a continuous function of $g$
and $X$ when $gX\in\tilde{M}$, the preimage of $D$ is an open set
of $(g,X)$. Thus the restriction is an open subset of $G\times\iota(\tilde{M})$,
and the restricted group action is continuous. It follows from this
that the equivalence relation is open. 
\end{proof}
\begin{lem}
\label{qbasis}An equivalence relation on $X$ is open if and only
if the map $\pi:X\rightarrow X/\sim,\, X\mapsto[X]$ is open. When
$\sim$ is open and $X$ has a countable basis of open sets, then
$X/\sim$ has a countable basis also.
\end{lem}
\begin{proof}
See lemma 2.3 in \cite{boothby75}. 
\end{proof}
\begin{lem}
\label{qhausdorff}Let $\sim$ be an open equivalence relation on
a topological space $X$. Then $R=\{(x,y)\mid x\sim y\}$ is a closed
subset of the space $X\times X$ if and only if the quotient space
$X/\sim$ is Hausdorff.
\end{lem}
\begin{proof}
See lemma 2.4 in \cite{boothby75}. 
\end{proof}
\begin{lem}
\label{pairsclosed}The set \[
\{(\iota(X),\iota(Y))\mid\exists g\,\iota(X)=\iota(gY)\, X,gY\in\tilde{M}\}\]
 is a closed subset of $\iota(\tilde{M})\times\iota(\tilde{M})$.
\end{lem}
\begin{proof}

Since $\iota^{2}=\iota$, we can rewrite an equation of the form $\iota(gX)=\iota(Y)$
into the form $\iota(gX)=Y$, where $Y_{3i}=1$. Then $gX_{4i}=X_{4i}$
so $\iota(gX)=Y$ is equivalent to \begin{equation}
\begin{matrix}(gX)_{1i}/(gX)_{3i}=Y_{1i}\end{matrix}\label{eq:6-10a}\end{equation}

\begin{equation}
(gX)_{2i}/(gX)_{3i}=Y_{2i}\label{eq:6-10b}\end{equation}

\begin{equation}
(gX)_{4i}/(gX)_{3i}=Y_{4i}\label{eq:6-10c}\end{equation}

which is equivalent to \begin{equation}
\begin{matrix}(gX)_{3i}Y_{4i}=X_{4i}\end{matrix}\label{eq:6-11a}\end{equation}

\begin{equation}
(gX)_{2i}Y_{4i}=Y_{2i}X_{4i}\label{eq:6-11b}\end{equation}

\begin{equation}
(gX)_{1i}Y_{4i}=Y_{1i}X_{4i}\label{eq:6-11c}\end{equation}

Denoting the columns of $X$ by $(x_{i},y_{i},z_{i},t_{i})$ and the
columns of $Y$ by $(x_{i}^{\prime},y_{i}^{\prime},1,t_{i}^{\prime})$
\begin{equation}
\begin{matrix}(g_{11}x_{i}+g_{12}y_{i}+g_{13}z_{i}+g_{14}t_{i})t_{i}^{\prime}=t_{i}x_{i}^{\prime}\end{matrix}\label{eq:6-12a}\end{equation}

\begin{equation}
(g_{21}x_{i}+g_{22}y_{i}+g_{23}z_{i}+g_{24}t_{i})t_{i}^{\prime}=t_{i}y_{i}^{\prime}\label{eq:6-12b}\end{equation}

\begin{equation}
(g_{31}x_{i}+g_{32}y_{i}+g_{33}z_{i}+g_{34}t_{i})t_{i}^{\prime}=t_{i}\label{eq:6-12c}\end{equation}

It suffices to prove that a sequence $X_{i},Y_{i}$ satisfying such
equations converges to $X,Y$ also satisfying these equations. There
is a unique solution to the equations $(gX)_{1i}Y_{4i}=Y_{1i}X_{4i}$
if and only if the matrix whose rows are $(x_{i}t_{i}^{\prime},y_{i}t_{i}^{\prime},z_{i}t_{i}^{\prime},t_{i}t_{i}^{\prime},t_{i}x_{i}^{\prime})$
has rank $4$. Since this matrix is the limit of a sequence in $\tilde{M}$,
where the limit is also in $\tilde{M}$, its rank is at least $4$.
Since the determinants of all $5\times5$ minors of the matrices in
the sequence are zero, the determinants of all $5\times5$ minors
of the limit matrix are zero. Hence the rank is exactly $4$. Similarly,
there is a unique solution to the equations $(gX)_{2i}Y_{4i}=Y_{2i}X_{4i}$.
There is always a unique solution to the equations $(gX)_{3i}Y_{4i}=X_{4i}$.
Thus we have shown the existence of $g$ such that $\iota(gX)=Y$.
It remains to prove that such a $g$ is non-singular. Since $\iota$
is invertible, $\iota(gX)=\iota(Y)$ implies $gX=Y$. If $det(g)=0$
then $Y\not\in\tilde{M}$, a contradiction. 
\end{proof}
\begin{thm}
The quotient space $\iota(\tilde{M})/\sim$ is a manifold, and is
metrisable.
\end{thm}
\begin{proof}
From lemma \ref{pairsclosed} and lemma \ref{qhausdorff} we see that
$\iota(\tilde{M})$ is a Hausdorff space. From lemma \ref{openrel}
and lemma \ref{qbasis} this quotient space also has a countable basis
of open sets. Given any $X\in\tilde{M}$, there is a neighbourhood
of the restricted orbit of $\iota(X)$ homeomorphic to a neighbourhood
of the orbit of $X$ in $G(3,k-1)$, because $G$ acts freely, and
$X\mapsto[\iota(X)]$ is continuous and open. Thus $\iota(\tilde{M})/\sim$
is locally euclidean. Coordinate neighbourhoods on $\iota(\tilde{M})/\sim$
will be compatible because the corresponding neighbourhoods on $G(3,k-1)$
are compatible. Hence $\iota(\tilde{M})/\sim$ is a differentiable
manifold. We can use the metric on $G(3,k-1)$ to give us a metric
on $\iota(\tilde{M})/\sim$. 
\end{proof}
\vskip 5pt It is actually the case that a metric exists on any manifold.
This is a deep theorem called the Urysohn metrisation theorem \cite{munkres75}.

\section{Conclusion}

What we have shown is that the stereo vision group acts freely. This
is almost an example of representational consistency, but not if there
are involutions. In general a matrix involution is a diagonal matrix
surrounded on both sides by an orthogonal matrix and its transpose,
with all diagonal entries $\pm1$ but not all $1$. Thus if the stereo
vision group $G$ is split up (and neural networks do this kind of
thing) it becomes an example of representational consistency. This
is exactly like a binary string descriptor. 

Roger Penrose argues \cite{penrose89} that by use of our consciousness
we are enabled to perform actions that lie beyond any kind of computational
activity. In the sequel \cite{penrose94} Penrose exhibits a tiling
problem that is non-computable. He argues that we perform non-computational
feats when we consciously understand. The example of the stereo vision
group is a construction in this direction. 

In my lectures \cite{bhav97} I begin with geometric optics, because
that is essentially all that one needs to assume. The rest is deduced
from mathematics. It is truly remarkable that optics can reveal something
like this. 

There is much more to be said on this subject. It was reached via
Kendall's shape spaces, which offers a beautiful and vivid demonstration
of representational consistency. The optics of the eye have been studied
since ancient times, and are being modelled in great detail for example
\cite{pierscionek2005}. There is a stronger definition of representational
consistency, which I have not covered either. For this automorphisms
of free groups are of interest, a recent article \cite{newman2008}.
And it seems possible to find representational consistency in biochemistry. 

\bibliographystyle{plain}
\bibliography{BASE,PHYSIOL}

\end{document}